\documentclass[11pt]{amsart}
\usepackage{graphicx,color,mathrsfs,amsmath,amssymb}

\theoremstyle{plain}

\theoremstyle{plain}
\newtheorem{theorem}{Theorem} [section]

\newtheorem{lemma}[theorem]{Lemma}
\newtheorem{proposition}[theorem]{Proposition}

\theoremstyle{definition}
\newtheorem{definition}[theorem]{Definition}

\theoremstyle{remark}

\numberwithin{theorem}{section}
\numberwithin{equation}{section}
\numberwithin{figure}{section}

\def\mean#1{\mathchoice
         {\mathop{\kern 0.2em\vrule width 0.6em height 0.69678ex depth -0.58065ex
                 \kern -0.8em \intop}\nolimits_{\kern -0.4em#1}}%
         {\mathop{\kern 0.1em\vrule width 0.5em height 0.69678ex depth -0.60387ex
                 \kern -0.6em \intop}\nolimits_{#1}}%
         {\mathop{\kern 0.1em\vrule width 0.5em height 0.69678ex
             depth -0.60387ex
                 \kern -0.6em \intop}\nolimits_{#1}}%
         {\mathop{\kern 0.1em\vrule width 0.5em height 0.69678ex depth -0.60387ex
                 \kern -0.6em \intop}\nolimits_{#1}}}

\def\N{\mathbb N}
\def\Z{\mathbb Z}

\def\R{\mathbb R}
\def\T{\mathbb T}
\def\L{\mathscr L}

\def\l{\lambda}

\newcommand{\re}{\mathbb{R}}
\newcommand{\n}{\mathbb{N}}
\newcommand{\z}{\mathbb{Z}}

\newcommand{\negint}{{\int\negthickspace\negthickspace\negthickspace-}}

\begin{document}

\title[Existence of Eulerian solutions of the semigeostrophic equations]{Existence of Eulerian solutions\\
to the semigeostrophic equations in physical space:\\ the 2-dimensional periodic case}

\author[L.\ Ambrosio]{Luigi Ambrosio}
\address{Scuola Normale Superiore,
p.za dei Cavalieri 7, I-56126 Pisa, Italy}
\email{l.ambrosio@sns.it}

\author[M.\ Colombo]{Maria Colombo}
\address{Scuola Normale Superiore,
p.za dei Cavalieri 7, I-56126 Pisa, Italy}
\email{maria.colombo@sns.it}

\author[G.\ De Philippis]{Guido De Philippis}
\address{Scuola Normale Superiore,
p.za dei Cavalieri 7, I-56126 Pisa, Italy}
\email{guido.dephilippis@sns.it}

\author[A.\ Figalli]{Alessio Figalli}
\address{Department of Mathematics,
The University of Texas at Austin, 1 University Station C1200,
Austin TX 78712, USA}
\email{figalli@math.utexas.edu}

%

\begin{abstract}
In this paper we use new regularity and stability
estimates for Alexandrov solutions to Monge-Amp\`ere equations,
recently estabilished by De Philippis and Figalli
\cite{DepFi}, to provide global in
time existence of distributional solutions to the semigeostrophic equations
on the 2-dimensional torus, under very mild assumptions on
the initial data. A link with Lagrangian solutions is also discussed.
\end{abstract}

\maketitle

\section{Introduction}
The semigeostrophic equations are a simple model used
in meteorology to describe large scale atmospheric flows.
As explained for instance in \cite[Section 2.2]{bebr} and \cite[Section 1.1]{Loe2}
(see also \cite{cu} for a more complete exposition),
the semigeostrophic equations can be derived from the 3-d incompressible
Euler equations, with Boussinesq and hydrostatic approximations, subject to a strong
Coriolis force.
Since for large scale atmospheric flows the Coriolis force dominates
the advection term, the flow is mostly bi-dimensional. For this reason,
the study of the semigeostrophic equations in 2-d or 3-d is pretty similar,
and in order to simplify our presentation we focus here on the 2-dimentional periodic case,
though we expect that our results could be extended to three dimensions.

The semigeostrophic system on the $2$-dimensional torus $\T^2$ is given by
\begin{equation}\label{eqn:SGsystem1}
\begin{cases}
\partial_t u^g_t(x) +\bigl(u_t(x) \cdot \nabla\bigr) u^g_t(x) + \nabla p_t(x) = -J u_t(x)
\quad\quad &(x,t)\in \T^2 \times (0,\infty)\\
u^g_t(x) = J \nabla p_t(x) & (x,t)\in \T^2 \times [0,\infty)\\
\nabla \cdot u_t(x) = 0  &(x,t)\in \T^2 \times [0,\infty)\\
p_0(x)= p^0(x)   &x\in \T^2.
\end{cases}
\end{equation}
Here $p^0$ is the initial datum, $J$ is the rotation matrix given by
\[
J:=
\begin{pmatrix}
0  &  -1   \\
1  &  0 \\
\end{pmatrix},
\]
and the functions $u_t$ and $p_t$ represent respectively the
\emph{velocity} and the \emph{pressure}, while $u^g_t$ is the
so-called \emph{semi-geostrophic} wind.\footnote{Note that we are
using the notation $u_t,\  p_t,\ u_t^g$ to denote the functions
$u(t,\cdot),\ p(t,\cdot),\ u^g(t,\cdot)$} Clearly the pressure is
defined up to a (time-dependent) additive constant. In the sequel we
are going to identify functions (and measures) defined on the torus
$\T^2$ with $\Z^2$-periodic functions defined on $\R^2$.

Substituting the relation $u^g_t = J \nabla p_t$
into the equation, the system \eqref{eqn:SGsystem1} can be rewritten
as
\begin{equation}\label{eqn:SGsystem2}
\begin{cases}
\partial_t J\nabla p_t + J \nabla^2p_t u_t +\nabla p_t +Ju_t = 0\\
\nabla \cdot u_t = 0 \\
p_0= p^0
\end{cases}
\end{equation}
with $u_t$ and $p_t$ periodic.

Energetic considerations (see \cite[Section 3.2]{cu}) show that it
is natural to assume that $p_t$ is ($-1$)-convex, i.e., the function
$P_t(x):=p_t(x)+|x|^2/2$ is convex on $\R^2$. If we denote with
$\L_{\T^2}$ the (normalized) Lebesgue measure on the torus, then
formally\footnote{Given a measure $\mu$ on $\T^2$ and a Borel map
$f:\T^2\to\T^2$, we define the measure $f_\sharp \mu $ through the
relation \[\int_{\T^2} h(y) \,d \, f_\sharp \mu(y)=\int_{\T^2}
h(f(x))\,d\mu(x)\]} $\rho_t:=(\nabla P_t)_\sharp \L_{\T^2}$ \
satisfies the following \emph{dual problem} (see the Appendix):

\begin{equation}\label{eqn:dualsystem}
\begin{cases}
\partial_t \rho_t +\nabla \cdot (U_t \rho_t) = 0 \\
U_t(x) = J(x-\nabla P_t^*(x))\\
\rho_t= (\nabla P_t)_{\sharp} \L_{\T^2}\\
P_0(x)= p^0(x)+|x|^2/2.
\end{cases}
\end{equation}
Here $P^*_t$ is the convex conjugate of $P_t$, namely
\[
P_t^*(y):=\sup_{x\in\R^2} (y \cdot x-P_t(x)).
\]
Notice that, since $P_t(x)-|x|^2/2$ is periodic,
\begin{equation}\label{eq:periodoP_t}
\nabla P_t(x+h)=\nabla P_t (x)+h\qquad\forall
x\in\R^2,\,\,\,h\in\z^2.
\end{equation}
Hence $\nabla P_t$ can be viewed as a map from $\T^2$ to $\T^2$ and
$\rho_t$ is a well defined measure on $\T^2$. One can also verify
easily that the inverse map $\nabla P_t^*$ satisfies
\eqref{eq:periodoP_t} as well. Accordingly, we shall understand
\eqref{eqn:dualsystem} as a PDE on $\T^2$, i.e., using
test functions which are $\Z^2$-periodic in space.

The dual problem \eqref{eqn:dualsystem} is nowadays pretty well
understood. In particular, Benamou and Brenier proved in
\cite{bebr} existence of weak solutions to \eqref{eqn:dualsystem},
see Theorem~\ref{thm:dualeq} below. On the contrary, much less is
known about the original system \eqref{eqn:SGsystem1}. Formally,
given a solution $\rho_t$ of \eqref{eqn:dualsystem} and defining
$P_t^*$ through the relation $\rho_t= (\nabla P_t)_\sharp\L_{\T^2}$
(namely the optimal transport map from $\rho_t$ to $\L_{\T^2}$, see
Theorem~\ref{thm:transport-torus}) the pair $(p_t,u_t)$ given
by\footnote{Because of the many compositions involved in this paper,
we use the notation $[\partial_t f](g)$
(resp. $[\nabla f](g)$) to denote the composition $(\partial_t
f)\circ g$ (resp. $(\nabla f)\circ g$), avoiding the ambiguous
notation $\partial_t f(g)$ (resp. $\nabla f(g)$)}
\begin{equation}\label{eqn:velocity}
\begin{cases}
p_t(x):=P_t(x)-|x|^2/2& \cr u_t(x):= [\partial_t\nabla P_t^*](\nabla
P_t(x)) + [\nabla^2 P_t^*]( \nabla P_t(x)) J(\nabla P_t(x) - x)&
\end{cases}
\end{equation}
solves \eqref{eqn:SGsystem2}. However, being $P^*_t$ just a convex
function, \emph{a priori} $\nabla^2 P_t^*$ is just a matrix-valued measure,
thus as pointed out in \cite{cufe} it is not clear the meaning to
give to the previous equation.

In this paper we prove that \eqref{eqn:velocity} is a well defined
velocity field, and that the couple $(p_t,u_t)$ is a solution of
(\ref{eqn:SGsystem1}) in a distributional sense. In order to carry
out our analysis, a fundamental tool is a recent result for
solutions of the Monge-Amp\`ere equation, proved by the third and
fourth author in \cite{DepFi}, showing $L\log^kL$ regularity on
$\nabla^2 P_t^*$ (see Theorem \ref{thm:transport-torus reg}(ii)
below).

Thanks to this result, we can easily show that the second term
appearing in the definition of the velocity $u_t$ in
\eqref{eqn:velocity} is a well defined $L^1$ function (see the proof
of Theorem \ref{thm:main}). Moreover, following some ideas developed
in \cite{Loe1} we can show that the first term is also $L^1$, thus
giving a meaning to $u_t$ (see Proposition~\ref{prop:est}). At this
point we can prove that the pair $(p_t,u_t)$ is actually a
distributional solution of system \eqref{eqn:SGsystem2}. Let us
recall, following \cite{cufe}, the proper definition of \emph{weak
Eulerian solution} of \eqref{eqn:SGsystem2}.

\begin{definition}\label{sg-weak-eul}
Let $p:\T^2\times (0,\infty)\to\R$ and $u:\T^2\times
(0,\infty)\to\R^2$. We say that $(p,u)$ is a \emph{weak Eulerian
solution} of \eqref{eqn:SGsystem2} if:
\begin{enumerate}
\item[-] $|u|\in L^\infty((0,\infty),L^1(\T^2))$,
$p\in L^\infty ((0,\infty),W^{1,\infty}(\T^2))$, and
$p_t(x)+|x|^2/2$ is convex for any $t \geq 0$;
\item[-] For every
$\phi\in C^\infty_c(\T^2\times [0,\infty))$, it holds
\begin{multline}\label{eqn:sg1-weak}
 \int_0^\infty\int_{\T^2}
 J\nabla p_t(x) \Big\{\partial_t \phi_t(x) + u_t(x)\cdot \nabla
 \phi_t(x)\Big\}-\Big\{\nabla p_t(x)+Ju_t(x) \Big\} \phi_t(x)\, dx \, dt\\
  +\int_{\T^2} J \nabla p_0(x) \phi_0(x) \, dx = 0;
\end{multline}
\item[-] For a.e. $t\in (0,\infty)$ it holds
\begin{equation}\label{eqn:sg2-weak}
\int_{\T^2 } \nabla \psi(x)\cdot u_t(x) \, dx =0 \qquad \mbox{for
all $\psi\in C^\infty(\T^2)$.}
\end{equation}
\end{enumerate}
\end{definition}

We can now state our main result.

\begin{theorem}\label{thm:main}
Let $p_0 : \re^2 \to \re$ be a $\z^2$-periodic function such that
$p_0(x)+ |x|^2/2$ is convex, and assume that the measure $
(Id+\nabla p_0 )_\sharp \L^2$ is absolutely continuous with
respect to $\L^2$ with density $\rho_0$, namely
\[
(Id+\nabla p_0)_\sharp\L^2 =\rho_0\L^2.
\]
Moreover, let us assume that both $\rho_0$ and $1/\rho_0$ belong to
$L^\infty(\R^2)$.

Let $\rho_t$ be the solution of \eqref{eqn:dualsystem} given by
Theorem~\ref{thm:dualeq} and let $P_t:\R^2\to \R$ be the (unique up
to an additive constant) convex function such that $(\nabla
P_t)_\sharp\L^2=\rho_t\L^2$ and $P_t(x)-|x|^2/2$ is $\Z^2$-periodic,
$P_t^*:\R^2\to\R$ its convex conjugate.

Then the couple $(p_t,u_t)$ defined in \eqref{eqn:velocity} is a weak Eulerian solution of
\eqref{eqn:SGsystem2}, in the sense of
Definition~\ref{sg-weak-eul}.
\end{theorem}

Although the vector field $u$ provided by the previous theorem is only $L^1$,
as explained in Section \ref{sect:RLF} we can associate to it a measure-preserving Lagrangian flow.
In particular we recover (in the particular case of the
$2$-dimensional periodic setting) the result of Cullen and Feldman
\cite{cufe} on the existence of Lagrangian solutions to the semigeostrophic equations in physical space.

The paper is structured as follows: in Section \ref{sect:MA}
we recall some preliminary results on optimal transport maps on the torus and their regularity.
Then, in Section \ref{sect:dual} we state the existence result of Benamou and Brenier
for solutions to the dual problem \eqref{eqn:dualsystem},
and we show some important regularity estimates on such solutions,
which are used in Section \ref{sect:proof thm}
 to prove Theorem \ref{thm:main}.
 In Section \ref{sect:RLF} we prove the existence of a ``Regular Lagrangian Flow''
 associated to the vector field $u$ provided by Theorem \ref{thm:main}.
 Finally, in Section \ref{sect:open pbs} we list some open problems.
 For completenes, in the Appendix we show the formal computation used to obtain \eqref{eqn:dualsystem}
 from \eqref{eqn:SGsystem2}.

\smallskip
\noindent {\bf Acknowledgement.} L.A., G.D.P., and A.F.
acknowledge the support of the ERC ADG GeMeThNES.
A.F. was also supported by the NSF Grant DMS-0969962.

\section{Optimal transport maps on the torus and their regularity}
\label{sect:MA}

The following theorem can be found in \cite{Cor} (see for instance \cite[Section 2]{DepFi} for the
notion of Alexandrov solution of the Monge-Amp\`ere equation).

\begin{theorem}[Existence of optimal maps on $\T^2$]\label{thm:transport-torus}
Let $\mu$ and $\nu$ be $\z^2$-periodic Radon measures on $\re^2$
such that $\mu([0,1)^2)=\nu([0,1)^2)=1$ and $ \mu=\rho \L^2$ with
$\rho>0$ almost everywhere. Then there exists a unique (up to an
additive constant) convex function $P:\re^2\to\re$ such that
$(\nabla P)_\sharp \mu = \nu$ and $P-|x|^2/2$ is $\Z^2$-periodic.
Moreover
\begin{equation}
\nabla P(x+h) = \nabla P(x) +h \quad\quad {\rm for \,a.e.}\,x\in\re^2,\,\,\,
\forall \, h\in \z^2,\label{ts:transp-per}
\end{equation}
\begin{equation}\label{diam-toro}
| \nabla P(x)-x |\leq {\rm diam}(\T^2) =
\frac{\sqrt{2}}{2}\qquad{\rm for \,a.e.}\, x\in\R^2.
\end{equation}
In addition, if $\mu=\rho\L^2$, $\nu=\sigma\L^2$, and there exist
constants $0<\lambda\leq\Lambda<\infty$ such that $\lambda\leq
\rho,\sigma\leq\Lambda$, then $P$ is a strictly convex
Alexandrov solution of
$$\det\nabla^2 P(x) = f(x),\qquad\text{with }f(x)=\frac{\rho(x)}{\sigma(\nabla P(x))}.$$
\end{theorem}
\begin{proof}Existence of $P$ follows from \cite{Cor}. To prove uniqueness we observe that, under our assumption,
also $p^*(y):=P^*(y)-|y|^2/2$ is $\Z^2$-periodic. Hence,
since
$$
P(x)=\sup_{y\in \R^2} x \cdot y - P^*(y),
$$
we get that the function $p(x):=P(x)-|x|^2/2 $ satisfies
\[
\begin{split}
p(x)&=\sup_{y\in \R^2}\Big(-\frac{|y-x|^2}{2}-P^*(y)+\frac{|y|^2}{2}\Big)\\
&=\sup_{y\in[0,1|^2}\sup_{h \in \Z^2}\Big(-\frac{|y+h-x|^2}{2}-p^*(y+h)\Big)\\
&=\sup_{y\in \T^2}\Big(-\frac{d^2_{\T^2}(x,y)}{2}-p^*(y)\Big),
\end{split}
\]
where $d_{\T^2}$ is the quotient distance on the torus, and we used
that $p^*(y)$ is $\Z^2$-periodic. This means that the function $p$
is $d_{\T^2}^2$-convex, and that $p^*$ is its $d^2_{\T^2}$-transform
(see \cite[Chapter 5]{Vil}). Hence $\nabla P=Id+\nabla p:\T^2 \to
\T^2$ is the unique ($\mu$-a.e.) optimal transport map sending $\mu$
onto $\nu$ (\cite[Theorem 9]{Mc}), and since $\rho>0$ almost
everywhere this uniquely characterizes $P$ up to an additive
constant. Finally, all the other properties of $P$ follow from
\cite{Cor}.
\end{proof}

Combining the previous theorem and the known regularity
results for strictly convex Alexandrov solutions of the
Monge-Amp\`ere equation (see \cite{CA1,CAW2p,CA2,Cor,DepFi,GT}) we
have the following:

\begin{theorem}[Space regularity of optimal maps on $\T^2$]\label{thm:transport-torus reg}
Let $\mu=\rho\L^2$, $\nu=\sigma\L^2$, and let $P$ be as in
Theorem~\ref{thm:transport-torus} with $\int_{\T^2}P\,dx=0$. Then:
\begin{enumerate}
\item[(i)] $P\in C^{1,\beta}(\T^2)$ for some $\beta=\beta(\lambda,\Lambda)\in (0,1)$, and there exists
a constant
$C=C(\lambda,\Lambda)$ such that
$$\| P\|_{C^{1,\beta}}\leq C.$$
\item[(ii)] $P\in W^{2,1}(\T^2)$, and for any $k \in \N$ there exists a constant $C=C(\l,\Lambda,k)$ such that
$$
\int_{\T^2}|\nabla^2 P| \log_+^k |\nabla^2 P | \,dx \leq C.
$$
\item[(iii)] If $\rho,\,\sigma\in C^{k,\alpha}(\T^2)$ for some $k\in\n$ and $\alpha\in(0,1)$, then
$P\in C^{k+2,\alpha}(\T^2)$ and there exists a constant
$C=C(\lambda,\Lambda,\|\rho\|_{C^{k,\alpha}},
\|\sigma\|_{C^{k,\alpha}}  )$  such that
$$\| P\|_{C^{k+2,\alpha}}\leq C.$$
Moreover, there exist two positive constants $c_1$ and $c_2$,
depending only on $\lambda$, $\Lambda$, $\|\rho\|_{C^{0,\alpha}}$,
and $\|\sigma\|_{C^{0,\alpha}}$, such that
$$c_1 Id \leq \nabla^2  {P}(x) \leq c_2 Id\qquad\forall \,x\in\T^2.$$
\end{enumerate}
\end{theorem}

\section{The dual problem and the regularity of the velocity field}\label{sect:dual}

In this section we recall some properties of solutions of
\eqref{eqn:dualsystem}, and we show the $L^1$ integrability
of the velocity field $u_t$ defined in \eqref{eqn:velocity}.

We know by Theorem \ref{thm:transport-torus} that $\rho_t$ uniquely
defines $P_t$ (and so also $P_t^*$) through the relation $(\nabla
P_t)_\sharp\L_{\T^2}=\rho_t$ up to an additive constant. We have the
following result (see \cite{bebr, cufe}):

\begin{theorem}[Existence of solutions of \eqref{eqn:dualsystem}]\label{thm:dualeq}
Let $P_0:\R^2\to\R$ be a convex function such that $P_0(x)-|x|^2/2$
is $\Z^2$-periodic, $(\nabla P_0)_\sharp\L_{\T^2}\ll\L^2$, and the
density $\rho_0$ satisfies $0<\lambda\le\rho_0\le\Lambda<\infty$.
Then there exist convex functions $P_t,\,P_t^* :\R^2\to\R$, with
$P_t(x)-|x|^2/2$ and $P_t^*(y)-|y|^2/2$ periodic, uniquely
determined up to time-dependent additive constants, such that
$(\nabla P_t)_\sharp \L^2=\rho_t\L^2$, $(\nabla P_t^*)_\sharp
\rho_t=\L_{\T^2}$. In addition, setting $U_t(x)=J(x-\nabla
P_t^*(x))$, $\rho_t$ is a distributional solution to
\eqref{eqn:dualsystem}, namely
\begin{equation}\label{eqn:sg-dual-weak}
 \int \int_{\T^2}
\Big\{ \partial_t \varphi_t(x) + \nabla \varphi_t(x) \cdot U_t(x)
\Big\} \rho_t(x)\, dx \, dt + \int_{\T^2} \varphi_0(x) \rho_0(x)
\, dx = 0
\end{equation}
for every $\varphi\in C^\infty_c(\R^2\times [0,\infty))$
$\Z^2$-periodic in the space variable.

Finally, the following regularity properties hold:
\begin{enumerate}
\item[(i)]$\lambda\le\rho_t\le\Lambda$;
\item[(ii)] $\rho_t\L^2 \in C([0,\infty), \mathcal P_w (\T^2))$;\footnote{Here
$\mathcal P_w(\T^2)$ is the space of probability measures on the
torus endowed with the \emph{weak} topology induced by the duality
with $C(\T^2)$}
\item[(iii)] $P_t-\negint_{\T^2}P_t,\, P^*_t-\negint_{\T^2}P_t^*
\in L^\infty([0,\infty),W_{\rm loc}^{1,\infty}(\R^2))\cap
C([0,\infty),W_{\rm loc}^{1,r}(\R^2))$ for every $r\in [1,\infty)$;
\item[(iv)] $\|U_t\|_\infty \leq \sqrt{2}/2$.
\end{enumerate}
\end{theorem}

To be precise, in \cite{bebr,cufe} the proof is given in $\re^3$,
but actually it can be rewritten verbatim on the $2$-dimensional
torus, using the optimal transport maps provided by
Theorem~\ref{thm:transport-torus}.
Observe that, by Theorem \ref{thm:dualeq}(ii), $t \mapsto \rho_t\L^2$ is weakly continuous,
so $\rho_t$ is a well-defined function \textit{for every} $t\geq 0$.

Further regularity properties of $\nabla P_t$ and $\nabla P_t^*$ with respect to
time will be proved in Propositions~\ref{prop:est} and \ref{prop:time-reg}.

In the proof of Theorem~\ref{thm:main} we will need to test  with
functions which are merely $W^{1,1}$. This is made possible by
the following lemma.

\begin{lemma}\label{rmk:scontr}
Let $\rho_t$ and $P_t$ be as in Theorem~\ref{thm:dualeq}. Then
\eqref{eqn:sg-dual-weak} holds for every $\varphi \in
W^{1,1}(\T^2\times [0,\infty))$ which is compactly supported in time.
(Now $\varphi_0(x)$ has to be understood in the sense of
traces.)
\end{lemma}
\begin{proof}
Let $\varphi^n\in C^\infty(\T^2\times [0,\infty))$ be strongly
converging to $\varphi$ in $W^{1,1}$, so that $\varphi_0^n$
converges to $\varphi_0$ in $L^1(\T^2)$. Taking into account that
both $\rho_t$ and $U_t$ are uniformly bounded from above in
$\T^2\times [0,\infty)$, we can apply
\eqref{eqn:sg-dual-weak} to the test functions $\varphi^n$
and let $n \to \infty$
to obtain the same formula with $\varphi$.
\end{proof}

The following proposition, which provides the Sobolev regularity of
$t\mapsto \nabla P_t^*$, is our main technical tool.
Notice that, in order to prove Theorem~\ref{thm:main}, only finiteness
of the left hand side in \eqref{ts:loep} would be needed, and the
proof of this fact involves only a smoothing argument, the
regularity estimates of \cite{DepFi} collected in
Theorem~\ref{thm:transport-torus reg}(ii), and the argument of
\cite[Theorem 5.1]{Loe1}. However, the continuity result in
\cite{DeFi2} allows to show the validity of the natural \emph{a priori} estimate on the
left hand side in \eqref{ts:loep}.

\begin{proposition}[Time regularity of optimal maps]
\label{prop:est} Let $\rho_t$ and $P_t$ be as in Theorem~\ref{thm:dualeq}. Then
$\nabla P_t^*\in W^{1,1}_{\rm loc}(\T^2\times [0,\infty);\R^2)$, and  for
every $k\in\n$ there exists a constant $C(k)$ such that, for almost every $t\geq 0$,
\begin{multline}
\label{ts:loep}
\int_{\T^2} \rho_t |\partial_t \nabla P_t^*| \log^k_+ (|\partial_t \nabla P_t^*|) \, dx\\
\leq C(k) \left( \int_{\T^2}\rho_t |\nabla^2 P_t^*|
\log^{2k}_+(|\nabla^2 P_t^*|) \, dx+ {\rm ess}\sup_{\T^2}\left(
\rho_t |U_t|^2 \right) \int_{\T^2} |\nabla^2 P_t^*| \, dx\right).
\end{multline}
\end{proposition}

To prove Proposition~\ref{prop:est}, we need some preliminary
results.

\begin{lemma}\label{lemma:orlicz}
For every $k\in\n$ we have
\begin{equation}
ab \log ^k_+(ab) \leq 2^{k-1} \left[ \left(\frac{k}{e}\right)^k +1 \right] b^2 + 2^{3(k-1)}a^2 \log^{2k}_+ (a) \quad\quad\forall \,(a,b)\in \re^+ \times \re^+.
\label{eqn:dis-num}
\end{equation}
\end{lemma}
\begin{proof}
From the elementary inequalities
$$ \log_+(ts) \leq \log_+(t)+\log_+(s),\quad (t+s)^k\le 2^{k-1} (t^k+s^k),\quad\log_+^k(t) \leq \left(\frac{k}{e}\right)^k t$$
which hold for every $t,\,s>0$, we infer
\[
\begin{split}
ab \log_+^k(ab)
&\leq ab\left[\log_+\left(\frac{b}{a}\right)+ 2\log_+(a) \right]^k\\
&\leq 2^{k-1}ab \left[\log_+^k\left(\frac{b}{a}\right)+2^k \log_+^k(a) \right]\\
&\leq 2^{k-1}\left[ \left(\frac{k}{e}\right)^k b^2+ 2^k ab\log_+^k(a) \right]\\
&\leq 2^{k-1}\left[ \left(\frac{k}{e}\right)^k b^2+ b^2+2^{2(k-1)}a^2 \log^{2k}_+ (a)\right],
\end{split}
\]
which proves \eqref{eqn:dis-num}.
\end{proof}

\begin{lemma}[Space-time regularity of transport]\label{lemma:MAlin}
Let $k \in \N\cup\{0\}$, and let $\rho\in C^\infty(\T^2\times [0,\infty))$
and $U\in
C^\infty(\T^2\times [0,\infty);\R^2)$ satisfy
$$0<\lambda \leq \rho_t(x) \leq \Lambda<\infty \quad\quad \forall\, (x,t)\in \T^2 \times [0,\infty),$$
$$\partial_t \rho_t +\nabla \cdot(U_t\rho_t) = 0 \quad\quad\text{in $\T^2\times [0,\infty),$}$$
and $\int_{\T^2}\rho_t\,dx=1$ for all $t\geq 0$. Let us consider
convex conjugate maps $P_t$ and $P_t^*$ such that $P_t(x)-|x|^2/2$ and
$P_t^*(y)-|y|^2/2$ are $\Z^2$-periodic, $(\nabla P_t^*)_\sharp
\rho_t=\L_{\T^2}$, $(\nabla P_t)_\sharp \L_{\T^2}=\rho_t$. Then:
\begin{enumerate}
\item[(i)] $P^*_t-\negint_{\T^2}P_t^* \in {\rm Lip_{\rm loc}} ([0,\infty);C^{k}(\T^2))$ for any $k \in \N$.
\item[(ii)] The following linearized Monge-Amp\`ere equation holds:
 \begin{equation}
 \nabla \cdot \bigl(\rho_t (\nabla^2P_t^*)^{-1}\partial_t \nabla P_t^*\bigr)  = - \nabla \cdot(\rho_t U_t).
 \label{ts:MAlin}
 \end{equation}
\end{enumerate}
\end{lemma}

\begin{proof} Let us fix $T>0$.
From the regularity theory for the Monge-Amp\`ere equation (see
Theorem~\ref{thm:transport-torus reg}) we obtain that $P_t\in
C^\infty(\re^2)$, uniformly for
$t \in [0,T]$, and there exist universal constants
$c_1,\,c_2>0$ such that
\begin{equation}
c_1 Id \leq \nabla^2 P_t^*(x) \leq c_2 Id\qquad\forall\, (x,t)\in\T^2\times[0,T].
\label{est:d2}
\end{equation}
Since $\nabla P_t^*$ is the inverse of $\nabla P_t$, by the smoothness of $P_t$
and \eqref{est:d2} we deduce that $P_t^*\in
C^\infty(\re^2)$, uniformly on
$[0,T]$.

Now, to prove (i), we need to investigate the time regularity of $P_t^*- \negint_{\T^2} P_t^*$.
Moreover, up to adding a time dependent constant to $P_t$, we can assume
without loss of generality that $\int_{\T^2} P^*_t=0$ for all $t$.
By the condition $(\nabla P_t^*)_\sharp \rho_t=\L_{\T^2}$ we get that for
any $0\le s,t\le T$ and $x\in \re^2$ it holds
\begin{equation} \label{eqn:rho-dif}
\begin{split}
\frac{ \rho_s(x) - \rho_t(x)}{s-t} &= \frac{ \det( \nabla^2P_s^*(x)) -\det( \nabla^2P_t^*(x))}{s-t}\\
&= \sum_{i,j=1}^2 \biggl(\int_0^1 \frac {\partial\det}{\partial \xi_{ij}} (
\tau \nabla^2  P_s^*(x)+ (1-\tau) \nabla^2 P_t^*(x)) \, d\tau\biggr) \,
\frac{ \partial _{ij} P_s^*(x) -\partial _{ij}P_t^*(x)}{s-t}.
\end{split}
\end{equation}
Given a $2\times 2$ matrix $A=(\xi_{ij})_{i,j=1,2}$, we denote by
$M(A)$ the cofactor matrix of $A$. We recall that
\begin{equation}\label{eqn:det-deriv}
 \frac {\partial\det (A)}{\partial \xi_{ij}}  = M_{ij}(A),
\end{equation}
and if $A$ is invertible then $M(A)$
satisfies the identity
\begin{equation}
 M(A)= \det(A) \, A^{-1}.
\label{eqn:cof}
\end{equation}
Moreover, if $A$ is symmetric and satisfies $c_1 Id \leq A\leq c_2
Id$ for some positive constants $c_1,\,c_2$, then
\begin{equation}\label{eqn:cofactor elliptic}
\frac{ c_1^2}{c_2} Id\leq M(A) \leq \frac{ c_2^2}{c_1} Id.
\end{equation}
Hence, from \eqref{eqn:rho-dif}, \eqref{eqn:det-deriv}, \eqref{est:d2}, and \eqref{eqn:cofactor
elliptic}, it follows that
\begin{equation}\label{eqn:dg-reg}
\frac{ \rho_s - \rho_t}{s-t} = \sum_{i,j=1}^2 \biggl(\int_0^1M_{ij}( \tau
\nabla^2  P_s^*+ (1-\tau) \nabla^2 P_t^*) \, d\tau\biggr) \, \partial _{ij}\biggl(\frac{
P_s^* -P_t^*}{s-t}\biggr),
 \end{equation}
with
\[
\frac{c_1^2}{c_2} Id \le \int_0^1 M_{ij} ( \tau \nabla^2 P_s^*+ (1-\tau)
\nabla^2 P_t^*) \, d\tau\le \frac{c_2^2}{c_1} Id
\]
Since $\nabla^2 P_t^*$ is smooth in space, uniformly on $[0,T]$,
by classical elliptic regularity theory\footnote{Note that equation \eqref{eqn:rho-dif} is well defined on $\T^2$
since $P^*_t-P^*_s$ is $\Z^2$-periodic. We also observe that $P^*_t-P^*_s$ has average
zero on $\T^2$.} it follows that for any $k \in \N$ and $\alpha \in (0,1)$
there exists
a constant $C:= C( \| (\rho_s - \rho_t) /(s-t) \|_{C^{k,\alpha}(\T^2
\times [0,T])})$ such that
\begin{equation*}
  \left\|\frac{ P_s^*(x) -P_t^*(x)}{s-t}\right\|_{C^{k+2,\alpha}(\T^2)}\leq C.
  \label{eqn:pot-est}
\end{equation*}
This proves point (i) in the statement. To prove the second part, we let $s \to t$ in
\eqref{eqn:dg-reg} to obtain
\begin{equation} \label{eqn:non-var}
\partial_t \rho_t= \sum_{i,j=1}^{2}M_{ij}( \nabla^2 P_t^*(x)) \,  \partial_t \partial_{ij} P_t^*(x).
\end{equation}
Taking into account the continuity equation and the well-known divergence-free property of the cofactor matrix
\[
\sum_{i} \partial_i M_{ij}( \nabla^2  {P_t}^*(x)) = 0,\qquad j=1,2,
\]
we can rewrite \eqref{eqn:non-var}  as
\[
-\nabla \cdot(U_t \rho_t) = \sum_{i,j=1}^{2}\partial_i\bigl(M_{ij}(
\nabla^2 P_t^*(x)) \,  \partial_t \partial_{j} P_t^*(x)\bigr).
\]
Hence, using \eqref{eqn:cof} and the Monge-Amp\`ere equation $\det(\nabla^2 P_t^*)=\rho_t$,
we finally get \eqref{ts:MAlin}.
\end{proof}

\begin{proof}[Proof of Proposition~\ref{prop:est}]
We closely follow the proof of \cite[Theorem 5.1]{Loe1}, and we split the proof in two parts.
In the first step we assume that
\begin{eqnarray}
&&\rho_t\in C^{\infty}(\T^2 \times \R ),\,\,U_t\in C^{\infty}(\T^2 \times \R;\R^2 )\,,
\label{eqn:loep-reg1}\\
&&0<\lambda \le \rho_t\le
\Lambda<\infty\,,\label{eqn:loep-reg2} \\
&&\partial_t
\rho_t+\nabla\cdot (U_t\rho_t)=0 \,,\label{eqn:loep-reg3}\\
&&
(\nabla P_t)_\sharp \L_{\T^2}=\rho_t\L_{\T^2},\label{eqn:loep-reg4}
\end{eqnarray}
and we prove that \eqref{ts:loep} holds for every $t\geq 0$.
In the second step we prove the general case through an approximation argument.

\medskip\emph{Step 1: The regular case.}
Let us assume that the regularity assumptions \eqref{eqn:loep-reg1},
\eqref{eqn:loep-reg2}, \eqref{eqn:loep-reg3}, \eqref{eqn:loep-reg4}
hold.
Moreover, up to adding a time dependent constant to $P_t$, we can assume
without loss of generality that $\int_{\T^2} P_t^*=0$ for all $t \geq 0$, so
that by Lemma~\ref{lemma:MAlin} we have $\partial_t P_t^* \in
C^\infty(\T^2)$. Fix $t\geq 0$.
Multiplying \eqref{ts:MAlin} by $\partial_t P_t^*$ and integrating
by parts, we get
\begin{equation}
\label{eqn:ma-div}
\begin{split}
 \int_{\T^2} \rho_t | (\nabla^2 P_t^*)^{-1/2}\partial_t \nabla P_t^*|^2\, dx
&=\int_{\T^2}  \rho_t \partial_t \nabla P_t^*\cdot (\nabla^2 P_t^*)^{-1} \partial_t \nabla P_t^* \,dx\\
& = -\int_{\T^2}  \rho_t  \partial_t \nabla P_t^* \cdot U_t \, dx.
\end{split}
\end{equation}
(Since the matrix $\nabla^2 {P_t}^*(x)$ is nonnegative, both its square root
and the square root of its inverse are well-defined.)
From Cauchy-Schwartz inequality it follows that the right-hand side of \eqref{eqn:ma-div}
can be rewritten and estimated with
\begin{equation}
\begin{split}\label{est:ma}
&-\int_{\T^2}  \rho_t  \partial_t \nabla P_t^*\cdot (\nabla^2 P_t^*)^{-1/2} (\nabla^2 P_t^*)^{1/2} U_t\,dx\\
& \leq \left( \int_{\T^2} \rho_t | (\nabla^2 P_t^*)^{-1/2}\partial_t
\nabla P_t^*|^2\, dx\right)^{1/2}  \left( \int_{\T^2}  \rho_t
|(\nabla^2 P_t^*)^{1/2}  U_t|^2\, dx \right)^{1/2}.
 \end{split}
\end{equation}
Moreover, the second factor in the right-hand side of \eqref{est:ma} can be estimated with
\begin{equation}\label{est:v}
\int_{\T^2}  \rho_t U_t \cdot \nabla^2 P_t^* U_t\,dx \leq
\max_{\T^2}\left( \rho_t |U_t|^2 \right) \int_{\T^2} |\nabla^2
P_t^*| \, dx.
\end{equation}
Hence, from \eqref{eqn:ma-div}, (\ref{est:ma}), and \eqref{est:v} it follows that
\begin{equation}\label{est:l2norms}
 \int_{\T^2} \rho_t | (\nabla^2 P_t^*)^{-1/2}\partial_t \nabla P_t^*|^2\, dx
 \leq\max_{\T^2}\left( \rho_t |U_t|^2 \right) \int_{\T^2} |\nabla^2 P_t^*|\, dx.
\end{equation}
We now apply Lemma~\ref{lemma:orlicz} with $a = |(\nabla^2
P_t^*)^{1/2}|$ and $b=  |(\nabla^2 P_t^*)^{-1/2}\partial_t \nabla
{P_t}^*(x)| $ to deduce the existence of a constant $C(k)$ such that
\begin{equation*} \begin{split}
|\partial_t \nabla P_t^*| \log_+^k(|\partial_t \nabla P_t^*|) &\leq
C(k)\left( |(\nabla^2 P_t^*)^{1/2}|^2 \log^{2k}_+(|(\nabla^2
P_t^*)^{1/2}|^2) +
| (\nabla^2 P_t^*)^{-1/2}\partial_t \nabla P_t^*|^2 \right)\\
&=C(k)\left( |\nabla^2 P_t^*|\log^{2k}_+(|\nabla^2 P_t^*|) +  |
(\nabla^2 P_t^*)^{-1/2}\partial_t \nabla P_t^*|^2\right).
\end{split}
\end{equation*}
Integrating the above inequality over $\T^2$ and using \eqref{est:l2norms}, we finally obtain
\begin{equation}\label{eqn: tsloepersmooth}
\begin{split}
&\int_{\T^2} \rho_t |\partial_t \nabla P_t^*| \log^k_{+} (|\partial_t \nabla P_t^*|) \, dx\\
&\leq C(k)\left( \int_{\T^2}\rho_t |\nabla^2 P_t^*| \log^{2k}_+(|\nabla^2 P_t^*|) \, dx
+ \int_{\T^2}\rho_t | (\nabla^2 P_t^*)^{-1/2}\partial_t \nabla P_t^*|^2\, dx\right)\\
&\leq C(k)\left( \int_{\T^2}\rho_t |\nabla^2 P_t^*|
\log^{2k}_+(|\nabla^2 P_t^*|) \, dx+ \max_{\T^2}\left( \rho_t
|U_t|^2 \right) \int_{\T^2} |\nabla^2 P_t^*|\,dx\right),
\end{split}
\end{equation}
which proves \eqref{ts:loep}.

\medskip
\emph{Step 2: The approximation argument.}
First of all, we extend the functions $\rho_t$ and $U_t$ for $t\leq 0$ by setting
$\rho_t=\rho_0$ and $U_t=0$ for every $t<0$. We notice that, with
this definition, $\rho_t$ solves the continuity equation with
velocity $U_t$ on $\re^2\times\re$.

Fix now $\sigma_1 \in C_c^{\infty}(\re^2)$, $\sigma_2 \in C_c^{\infty}(\re)$, define
the family of mollifiers $(\sigma^n)_{n\in \N}$ as $\sigma^n(x,t) := n^3\sigma_1(nx) \sigma_2(nt)$, and set
$$\rho^n:= \rho\ast\sigma^n, \quad\quad U^n(x) := \frac{(\rho U) \ast \sigma^n}{\rho \ast \sigma^n}. $$
Since $\lambda \leq \rho\leq \Lambda $ then
$$\lambda\leq\rho^n\leq \Lambda. $$
Therefore both $\rho^n$ and $U^n$ are well defined and satisfy
\eqref{eqn:loep-reg1}, \eqref{eqn:loep-reg2}, \eqref{eqn:loep-reg3}.
Moreover for every $t>0$ the function $\rho^n_t$ is $\z^2$-periodic
and it is a probability density  when restricted to $(0,1)^2$ (once
again we are identifying periodic functions with functions defined
on the torus). Let $P^n_t$ be the only convex function such that
$(\nabla P^n_t)_\sharp \L^2=\rho^n_t$ and its
its convex conjugate $P^{n*}_t$ satisfies $\int_{\T^2}P^{n*}_t=0$ for all $t \geq 0$.
Since $\rho^n_t \to \rho_t$ in
$L^1(\T^2)$ for any $t>0$ (recall that, by Theorem
\ref{thm:dualeq}(ii), $\rho_t$ is weakly continuous in time), from
standard stability results for Alexandrov solutions of Monge-Amp\`ere (see
for instance \cite{DeFi2}) it follows that
\begin{equation}\label{eqn:transp-conv}
\nabla P^{n*}_t \to \nabla P_t^* \quad\quad\mbox{in }L^1(\T^2)
\end{equation}
for any $t>0$. Moreover, by Theorems \ref{thm:transport-torus} and
\ref{thm:transport-torus reg}(ii), for every $k\in\N$ there exists a constant $C:=
C(\lambda, \Lambda,k)$ such that
$$
\int_{\T^2}\rho^n_t |\nabla^2 P^{n*}_t| \log^k_+(|\nabla^2
P^{n*}_t|) \, dx  \leq C,
$$
and by the stability theorem in the Sobolev topology
estabilished in \cite[Theorem 1.3]{DeFi2} it follows that
\begin{equation}\label{eqn:d2-conv-1}
 \int_{\T^2}\rho^n_t |\nabla^2 P^{n*}_t| \log^k_+(|\nabla^2 P^{n*}_t|) \, dx
 \to  \int_{\T^2}\rho_t |\nabla^2 P_t^*| \log^k_+(|\nabla^2 P_t^*|) \, dx,
 \end{equation}
\begin{equation}\label{eqn:d2-conv-2}
 \int_{\T^2} |\nabla^2 P^{n*}_t| \, dx \to  \int_{\T^2} |\nabla^2 P_t^*| \, dx .
\end{equation}
Finally, since the function $(w,t)\mapsto F(w,t)=|w|^2/t$ is convex
on $\R^2\times (0,\infty)$, by Jensen inequality we get
\begin{equation}\label{eqn:d2-conv-3}
\| \rho^n |U^n|^2\|_\infty=\|F(\rho^n U^n,\rho^n)\|_\infty \leq \|
\rho |U|^2\|_\infty.
\end{equation}
Let us fix $T>0$ and $\phi \in C^{\infty}_c((0,T))$ nonnegative. From the previous steps and Dunford-Pettis Theorem, it is clear that
$\phi(t)\rho^n_t \partial_t \nabla P^{n*}_t$ weakly converge
to $\phi(t)\rho_t\partial_t \nabla P_t^*$ in $L^1(\T^2\times (0,T))$. Moreover, since the function
$w\mapsto |w|\log_+^k(|w|/r)$ is convex for every $r\in (0,\infty)$
 we can apply Ioffe lower semicontinuity theorem \cite[Theorem 5.8]{AFP} 
 to the functions $\phi(t)\rho^n_t \partial_t \nabla P^{n*}_t$ and $\phi(t)\rho^n_t$ 
 to infer
\begin{equation}\label{eqn:time der}
\begin{split}
\int_0^T \phi(t)\int_{\T^2} \rho_t |\partial_t \nabla P_t^*| \log^k_{+} (|\partial_t \nabla P_t^*|) \, dx \, dt
\leq \liminf_{n\to\infty} \int_0^T \phi(t) \int_{\T^2} \rho^n_t
|\partial_t \nabla P^{n*}_t| \log^k_{+} (|\partial_t \nabla
P^{n*}_t|) \, dx \, dt.
\end{split}
\end{equation}
By Step 1 we can apply \eqref{ts:loep} to $\rho_t^n,U_t^n$. Taking \eqref{eqn:d2-conv-1},  \eqref{eqn:d2-conv-2},
\eqref{eqn:d2-conv-3} and \eqref{eqn:time der} into account, by Lebesgue dominated convergence theorem we obtain 
\begin{multline*}
\int_0^T\phi(t)\int_{\T^2} \rho_t |\partial_t \nabla P_t^*| \log^k_+ (|\partial_t \nabla P_t^*|) \, dx \, dt\\
\leq C(k) \int_0^T \phi(t) \left( \int_{\T^2}\rho_t |\nabla^2 P_t^*|
\log^{2k}_+(|\nabla^2 P_t^*|) \, dx+ {\rm ess}\sup_{\T^2}\left(
\rho_t |U_t|^2 \right) \int_{\T^2} |\nabla^2 P_t^*| \, dx\right) \, dt.
\end{multline*}
Since this holds for every $\phi \in C^{\infty}_c((0,T))$ nonnegative, we obtain the desired result.
\end{proof}

It is clear from the proof of Proposition~\ref{prop:est} that the
particular coupling between the velocity field $U_t$ and the
transport map $P_t$ is not used. Actually, using
Theorem~\ref{thm:transport-torus reg}(ii) and
\cite[Theorem~1.3]{DeFi2}, and arguing again as in the proof of
\cite[Theorem~5.1]{Loe1}, the following more general statement holds
(compare with \cite[Theorem 5.1, Equations (27) and (29)]{Loe1}):

\begin{proposition}\label{prop:time-reg}
Let $\rho_t$ and $v_t$ be such that $0<\lambda\leq \rho_t \leq
\Lambda<\infty$, $v_t \in L^\infty_{\rm loc}(\T^2 \times
[0,\infty),\R^2)$, and
 $$\partial_t \rho_t+\nabla\cdot (v_t\rho_t)=0.$$
Assume that $\int_{\T^2}\rho_t\,dx=1$ for all $t\geq 0$, let $P_t$
be a convex function such that
$$(\nabla P_t)_\sharp \L_{\T^2} = \rho_t \L_{\T^2} ,$$
and denote by $P^*_t$ its convex conjugate.

Then  $\nabla P_t$ and $\nabla P_t^*$ belong to $W^{1,1}_{\rm
loc}(\T^2 \times [0,\infty);\R^2)$. Moreover, for every
$k\in \N$ there exists a constant $C(k)$  such that, for almost every $t \geq 0$,
\begin{multline}\label{eqn:3.61}
\int_{\T^2} \rho_t |\partial_t \nabla P_t^*| \log^k_+ (|\partial_t \nabla P_t^*|) \, dx\\
\leq C(k) \left( \int_{\T^2}\rho_t |\nabla^2 P_t^*|
\log^{2k}_+(|\nabla^2 P_t^*|) \, dx+ {\rm ess}\sup_{\T^2}\left(
\rho_t |v_t|^2 \right) \int_{\T^2} |\nabla^2 P_t^*| \, dx\right),
\end{multline}
\begin{multline}\label{eqn:3.62}
\int_{\T^2} |\partial_t \nabla P_t| \log^k_+ (|\partial_t \nabla P_t|) \, dx\\
\leq C(k) \left( \int_{\T^2} |\nabla^2 P_t|
\log^{2k}_+(|\nabla^2 P_t|) \, dx+ {\rm ess}\sup_{\T^2}\left(
\rho_t |v_t|^2 \right) \int_{\T^2} |\nabla^2 P_t^*| \, dx\right).
\end{multline}
\end{proposition}

\begin{proof}
We just give a short sketch of the proof. Equation \eqref{eqn:3.61}
can be proved following the same line of the proof of Proposition
\ref{prop:est}. To prove \eqref{eqn:3.62} notice that by the
approximation argument in the second step of the proof of
Proposition~\ref{prop:est} we can assume that the velocity and the
density are smooth and hence, arguing as in Lemma~\ref{lemma:MAlin},
we have that $P_t,\, P_t^* \in {\rm Lip}_{\rm loc}([0,\infty)
,C^\infty(\T^2))$. Now, changing variables in the   the left hand
side of \eqref{est:l2norms} we get
\begin{equation}\label{pippo}
\int_{\T^2} \left|\bigl([\nabla^2 P_t^*](\nabla P_t)\bigr)^{-1/2}[\partial_t\nabla
P_t^*](\nabla P_t)\right|^2\, dx
 \leq\max_{\T^2}\left( \rho_t |v_t|^2 \right) \int_{\T^2} |\nabla^2 P_t^*|\, dx.
\end{equation}
Taking into account the identities
\[
[\nabla^2 P_t^*](\nabla P_t)=\big(\nabla^2 P_t\big)^{-1}
\qquad\text{and}\qquad
[\partial_t \nabla P_t^*](\nabla P_t)+[\nabla^2 P_t^*](\nabla P_t)\partial_t \nabla P_t=0
\]
which follow differentiating with respect to time and space  $\nabla
P^*_t\circ \nabla P_t =Id$, Equation \eqref{pippo} becomes
\[
\int_{\T^2}  | (\nabla^2 P_t)^{-1/2}\partial_t \nabla P_t|^2\, dx
 \leq\max_{\T^2}\left( \rho_t |v_t|^2 \right) \int_{\T^2} |\nabla^2 P_t^*|\, dx.
\]
At this point the proof of \eqref{eqn:3.62} is obtained arguing as in Proposition \ref{prop:est}.
\end{proof}

\section{Existence of an Eulerian solution}\label{sect:proof thm}

In this section we prove Theorem~\ref{thm:main}.

\begin{proof}[Proof of Theorem~\ref{thm:main}]
First of all notice that, thanks to Theorem~\ref{thm:transport-torus
reg}(i) and Proposition~\ref{prop:est}, it holds $|\nabla^2
P_t^*|,\,|\partial_t \nabla P^*_t|\in L_{\rm
loc}^\infty([0,\infty),L^1(\T^2))$. Moreover, since $(\nabla
P_t)_\sharp \L_{\T^2}=\rho_t\L_{\T^2}$, it is immediate to check the
function $u$ in \eqref{eqn:velocity} is well-defined\footnote{Note
that the composition of $\nabla^2P_t^*$ with $\nabla P_t$ makes
sense. Indeed, by the conditions $(\nabla P_t)_\sharp \L_{\T^2} =
\rho_t \L_{\T^2}\ll\L_{\T^2}$, if we change the value of
$\nabla^2P_t^*$ in a set of measure zero, also $[\nabla^2P_t^*](\nabla
P_t)$ will change only on a set of measure zero.} and $|u|$ belongs
to $L_{\rm loc}^\infty([0,\infty),L^1(\T^2))$.

Let $\phi\in C^\infty_c(\re^2\times[0,\infty))$ be a $\z^2$-periodic
function in space and let us consider the function $\varphi:
\re^2\times [0,\infty)\to\re^2$ given by
\begin{equation}\label{defn:test}
\varphi_t(y) := J(y-\nabla P_t^*(y)) \phi_t(\nabla P_t^*(y)).
\end{equation}
By Theorem \ref{thm:transport-torus} and the periodicity of $\phi$,
$\varphi_t(y)$ is $\z^2$-periodic in the space variable. Moreover
$\varphi_t$ is compactly supported in time, and
Proposition~\ref{prop:est} implies that $\varphi\in
W^{1,1}(\re^2\times[0,\infty))$. So, by Lemma~\ref{rmk:scontr},
each component of the function $\varphi_t(y)$ is an admissible test
function for \eqref{eqn:sg-dual-weak}. For later use, we write down
explicitly the derivatives of $\varphi$:
\begin{equation}\label{eqn:ber1}
\begin{cases}
\partial_t\varphi_t(y)=-J[\partial_t \nabla P_t^*](y)\phi_t(\nabla P_t^*(y))+
J(y-\nabla P_t^*(y))[\partial_t\phi_t](\nabla P_t^*(y))+\\
\qquad \,\,\,\,\,\,\,\,\,\,\,\,\,+ J(y-\nabla
P_t^*(y))\bigl([\nabla\phi_t](P_t^*(y))\cdot
\partial_t\nabla P_t^*(y)\bigr),
\\
\nabla\varphi_t(y)=J(Id-\nabla^2P_t^*(y))\phi_t(\nabla P_t^*(y))+
J(y-\nabla P_t^*(y))\otimes \bigl([\nabla^T\phi_t](P_t^*(y))\nabla^2P_t^*(y)\bigr).
\end{cases}
\end{equation}
Taking into account that $(\nabla P_t)_\sharp \L_{\T^2}=\rho_t\L_{\T^2}$
and that $[\nabla P_t^*](\nabla P_t(x))=x$ almost everywhere, we can rewrite the
boundary term in \eqref{eqn:sg-dual-weak} as
\begin{equation}
\label{eqn:test-1} \int_{\T^2} \varphi_0(y) \rho_0(y) \, dy
 =  \int_{\T^2} J (\nabla {P_0}(x)-x)\phi_0(x) \, dx\\
 =  \int_{R^2} J \nabla {p_0}(x)\phi_0(x) \, dx.
\end{equation}
In the same way, since $U_t(y)=J(y-\nabla P_t^*(y))$, we can use
\eqref{eqn:ber1} to rewrite the other term as
\begin{equation}
\label{eqn:test-12}
\begin{split}
&\int_0^\infty \int_{\T^2} \Big\{ \partial_t \varphi_t(y) +
\nabla \varphi_t(y) \cdot U_t(y)\Big\} \rho_t(y)\, dy \, dt\\
& = \int_0^\infty  \int_{\T^2}\Big\{ -J[\partial_t \nabla P_t^*](
\nabla P_t(x))\phi_t(x) + J( \nabla P_t(x) - x)\partial_t\phi_t(x)\\
&
+J( \nabla P_t(x) - x) \bigl(\nabla\phi_t(x)\cdot [\partial_t \nabla P_t^*]( \nabla P_t(x))\bigr)\\
&+ \bigl[J (Id-\nabla^2 P_t^*( \nabla P_t(x))) \phi_t(x) + J( \nabla
P_t(x) - x) \otimes \bigl(\nabla^T\phi_t(x)\nabla^2 P_t^*( \nabla
P_t(x))\bigr)\bigr]J(\nabla P_t(x)- x) \Big\}\, dx \, dt
\end{split}
\end{equation}
which, taking into account the formula \eqref{eqn:velocity} for $u$,
after rearranging the terms turns out to be equal to
\begin{equation}\label{eqn:test-2}
\int_0^{\infty} \int_{\T^2} \left\{ J\nabla p_t(x) \bigl(
\partial_t \phi_t(x) + u_t(x)\cdot \nabla \phi_t(x)\bigr)
+\bigl(-\nabla p_t(x)-Ju_t(x) \bigr) \phi_t(x)\right\}\, dx \, dt.
\end{equation}
Hence, combining (\ref{eqn:test-1}), \eqref{eqn:test-12},
(\ref{eqn:test-2}), and (\ref{eqn:sg-dual-weak}), we obtain the validity of
(\ref{eqn:sg1-weak}).

Now we prove (\ref{eqn:sg2-weak}). Given $\phi\in
C_c^\infty(0,\infty)$ and a $\z^2$-periodic function $\psi\in
C^\infty(\R^2)$, let us consider the function $\varphi:\re^2\times
[0,\infty)\to \re$ defined by
\begin{equation}
\varphi_t(y) := \phi(t)\psi(\nabla {P_t}^*(y)).
\label{defn:test-div}
\end{equation}
As in the previous case, we have that $\varphi$ is $\z^2$-periodic
in the space variable and $\varphi\in W^{1,1}(\T^2\times [0,\infty))$, so
we can use $\varphi$ as a test function in
(\ref{eqn:sg2-weak}). Then, identities analogous to
\eqref{eqn:ber1} yield
\begin{equation*}
\begin{split}
0& = \int_0^{\infty}\int_{\T^2} \left\{ \partial_t \varphi_t(y) +
\nabla \varphi_t(y) \cdot U_t(y)\right\} \rho_t(y)\, dy \, dt\\
& =\int_0^{\infty} \phi'(t)\int_{\T^2} \psi(x)\, dx \, dt\\
&\phantom{A}
 +  \int_0^{\infty} \phi(t)
 \int_{\T^2}\Big\{\nabla\psi(x)\cdot \partial_t \nabla {P_t}^*( \nabla P_t(x))
  +\nabla^T \psi(x)\nabla^2 P_t^*( \nabla P_t(x))J( \nabla P_t(x) - x) \Big\}\, dx \, dt\\
& = \int_0^{\infty} \phi(t) \int_{\T^2} \nabla \psi(x)\cdot u_t(x)
\, dx \, dt.
\end{split}
\end{equation*}
Since $\phi$ is arbitrary we obtain
$$
\int_{\T^2 } \nabla \psi(x)\cdot u_t(x) \, dx =0 \qquad \mbox{for
a.e. $t>0$.}
$$
By a standard density argument it follows that the above equation holds
outside a negligible set of times
independent of the test function $\psi$, thus proving \eqref{eqn:sg2-weak}.
\end{proof}

\section{Existence of a Regular Lagrangian Flow for the semigeostrophic velocity field}
\label{sect:RLF}

We start with the definition of Regular Lagrangian Flow for a given
vector field $b$, inspired by \cite{AM1,AM2}:

\begin{definition}\label{def:wl}
Given a Borel, locally integrable vector field $b:\T^2\times (0,\infty)\to\R^2$, we say
that a Borel function $F:\T^2\times [0,\infty)\to\T^2$ is a
\emph{Regular Lagrangian Flow (in short RLF) associated to $b$} if the following
two conditions are satisfied.
\begin{enumerate}
\item[(a)] For almost every $x\in\T^2$ the map $t\mapsto F_t(x)$ is locally
absolutely continuous in $[0,\infty)$ and
\begin{equation} \label{def:wl1}
F_t(x) =x+\int_0^t b_s(F_s(x))dx \quad \quad \forall t\in
[0,\infty).
\end{equation}
\item[(b)] For every $t\in [0,\infty)$ it holds $(F_t)_\#\L_{\T^2}\leq
C\L_{\T^2}$, with $C\in [0,\infty)$ independent of $t$.
\end{enumerate}
\end{definition}

A particular class of RLFs is the collection of
the measure-preserving ones, where (b) is strengthened to
\[
(F_t)_\#\L_{\T^2}=\L_{\T^2}\qquad\forall t\geq 0.
\]
Notice that \emph{a priori} the above definition depends on the choice of
the representative of $b$ in the Lebesgue equivalence class, since
modifications of $b$ in Lebesgue negligible sets could destroy
condition (a). However, a simple argument based on Fubini's theorem
shows that the combination of (a) and (b) is indeed invariant (see
\cite[Section~6]{AM1}): in other words, if $b=\widetilde b$ a.e. in
$\T^2\times (0,\infty)$, then every RLF
associated to $b$ is also a RLF associated to
$\widetilde b$.

We show existence of a measure-preserving RLF associated
to the vector field $u$ defined by
\begin{equation}\label{eqn:velocity1}
u_t(x)=[\partial_t \nabla P^*_t](\nabla P_t(x))
+[\nabla^2P^*_t](\nabla P_t(x))J(\nabla P_t(x)-x),
\end{equation}
where $P_t$ and $P_t^*$ are as in Theorem~\ref{thm:main}. Recall
also that, under these assumptions, $|u|\in L^\infty_{\rm
loc}([0,\infty),L^1(\T^2))$.

Existence for weaker notion of Lagrangian flow of the
semigeostrophic equations was proved by Cullen and Feldman, see
\cite[Definition~2.4]{cufe}, but since at that time the results of
\cite{DepFi} were not available the velocity could not be defined,
not even as a function. Hence, they had to adopt a more indirect
definition. We shall prove indeed that their flow is a flow
according to Definition~\ref{def:wl}. We discuss the uniqueness
issue in the last section.

\begin{theorem}
\label{thm:RLF}
Let us assume that the hypotheses of Theorem~\ref{thm:main} are
satisfied, and let $P_t$ and $P_t^*$ be the convex functions such
that
\[
(\nabla P_t)_\sharp \L_{\T^2}=\rho_t \L_{\T^2}, \qquad (\nabla
P^*_t)_\sharp\rho_t \L_{\T^2}=\L_{\T^2}.
\]
Then, for $u_t$ given by \eqref{eqn:velocity1} there exists a
measure-preserving RLF $F$ associated to $u_t$. Moreover
$F$ is invertible in the sense that for all $t\geq 0$ there exist
Borel maps $F^*_t$ such that $F^*_t(F_t)=Id$ and $F_t(F^*_t)=Id$
a.e. in $\T^2$.
\end{theorem}
\begin{proof}
Let us consider the velocity field in the dual variables $U_t(x)=J(x-\nabla P_t^*(x))$.
Since $P_t^*$ is convex, $U_t \in BV(\T^2;\R^2)$ uniformly in time (actually, by Theorem \ref{thm:transport-torus reg}(ii)
$U_t \in W^{1,1}(\T^2;\R^2)$). Moreover $U_t$ is divergence-free.
Hence, by the theory of Regular Lagrangian Flows associated to $BV$ vector fields
\cite{AM1,AM2}, there exists a unique\footnote{
The uniqueness of Regular Lagrangian Flows has to be understood in the following way:
if $G_1,G_2:\T^2\times [0,\infty)\to\T^2$ are two RLFs associated to $U$, then
the integral curves $G_1(\cdot,x)$ and $G_2(\cdot,x)$ are equal for $\L^2$-a.e. $x$.}
measure-preserving RLF $G:\T^2\times [0,\infty)\to\T^2$ associated to $U$.

We now define\footnote{Observe that the definition of $F$ makes sense.
Indeed, by Theorem \ref{thm:transport-torus reg}(i),
both maps $\nabla P_0$ and $\nabla P^*_t$ are H\"older continuous in space. Morever, by the weak continuity
in time of $t \mapsto \rho_t$ (Theorem \ref{thm:dualeq}(ii))
and the stability results for Alexandrov solutions of Monge-Amp\`ere, $\nabla P^*$ is continuous both in space and time.
Finally, since $(\nabla P_0)_\sharp \L_{\T^2} \ll\L_{\T^2}$, if we change the value of
$G$ in a set of measure zero, also $F$ will change only on a set of measure zero.
}
\begin{equation}
\label{eq:def F}
F_t(y):= \nabla P^*_t(G_t(\nabla P_0(y))).
\end{equation}
The validity of property (b) in Definition \ref{def:wl} and the invertibility of $F$
follow from the same arguments of
\cite[Propositions 2.14 and 2.17]{cufe}.
Hence we only have
to show that property (a) in Definition~\ref{def:wl} holds.

Let us define $Q^n:= B\ast\sigma^n$, where $B$ is a Sobolev and uniformly continuous extension of
$\nabla P^*$ to $\T^2\times\R$, and $\sigma^n$ is a standard family
of mollifiers in $\T^2\times\R$.  It is well known that $Q^n \to \nabla P^*$
locally uniformly and
in the strong topology of $W_{\rm loc}^{1,1}(\T^2\times[0,\infty))$. Thus,
using the measure-preserving property of $G_t$, for all $T>0$ we get
\begin{equation*}
\begin{split}
0 &= \lim_{n\to \infty} \int_{\T^2}  \int_0^T  \Big\{ | Q^n_t- \nabla P_t^*
|
+ |\partial_t Q^n_t - \partial_t \nabla P_t^*| + |\nabla Q^n_t - \nabla^2 P_t^*| \Big\} \, dy\, dt.\\
&= \lim_{n\to \infty} \int_{\T^2}  \int_0^T \Big\{ | Q^n_t(G_t) -
\nabla P_t^*(G_t) | + |[\partial_t Q^n_t](G_t) - [\partial_t \nabla P_t^*](G_t) | +
|[\nabla Q^n_t](G_t) - [\nabla^2 P_t^*](G_t) | \Big\} \, dx\, dt.
\end{split}
\end{equation*}
Up to a (not re-labeled) subsequence the previous convergence is
pointwise in space, namely, for almost every $x\in \T^2$,
\begin{equation}
\begin{split}
\int_0^T \Big\{ | Q^n_t(G_t(x)) - \nabla P_t^*(G_t(x)) | + |[\partial_t
Q^n_t](G_t(x)) &- [\partial_t \nabla P_t^*](G_t(x)) |\\
 &+ |[\nabla
Q^n_t](G_t(x)) - [\nabla^2 P_t^*](G_t(x)) | \Big\} \, dt\rightarrow 0.
\end{split}
\label{eqn:conv-x-fix}
\end{equation}
Hence, since $G$ is a RLF and by assumption
\[
(\nabla P_0)\L_{\T^2}\ll \L_{\T^2},
\]
for almost every $y$ we have that \eqref{eqn:conv-x-fix} holds at
$x=\nabla P_0(y)$, and
the function $t\mapsto G_t(x)$ is absolutely
continuous on $[0,T]$, with derivative given by
\begin{equation*}
 \frac{d}{dt} G_t(x) = U_t(G_t(x))=J(G_t(x)-\nabla P_t^*(G_t(x))) \qquad  \mbox{for a.e. }t\in[0,T].
 \end{equation*}
Let us fix such an $y$. Since $Q^n$ is smooth, the function
$Q^n_t(G_t(x))$ is absolutely continuous in $[0,T]$ and its time
derivative is given by
\begin{equation*}
\frac{d}{dt} \bigl(Q^n_t(G_t(x))\bigr) = [\partial_t Q^n_t](G_t(x))+[\nabla
Q^n_t](G_t(x)) J(G_t(x)-\nabla P_t^*(G_t(x))).
\end{equation*}
Hence, since $J(G_t(x)-\nabla P_t^*(G_t(x)))=U(G_t(x))$ is uniformly bounded, from
\eqref{eqn:conv-x-fix} we  get
\begin{equation}\label{eqn:deriv-lim}
\begin{split}
\lim_{n\to \infty} \frac{d}{dt} \bigl(Q^n_t(G_t(x))\bigr) &= [\partial_t
\nabla P_t^*](G_t(x)) +  [
\nabla^2 P_t^*](G_t(x))J(G_t(x)-\nabla P_t^*(G_t(x))):=v_t(y)\qquad\text{in $L^1(0,T)$.}
\end{split}
\end{equation}
Recalling that
$$
\lim_{n\to\infty} Q^n_t(G_t(x))=\nabla P_t^*(G_t(x)) =F_t(y) \qquad\forall \,t \in[0,T],
$$
we infer that $F_t(y)$ is absolutely continuous in $[0,T]$ (being the
limit in $W^{1,1}(0,T)$ of absolutely continuous maps). Moreover, by taking
the limit as $n \to \infty$ in the identity
$$
Q^n_t(G_t(x)) = Q^n_0(G_0(x)) + \int_0^t
\frac{d}{d\tau} \bigl(Q^n_\tau(G_\tau(x))\bigr) \, d\tau,
$$
thanks to
\eqref{eqn:deriv-lim} we get
\begin{equation}\label{eqn:f-ac}
\begin{split}
F_t(y) &= F_0(y) + \int_0^t v_\tau(y) \,d \tau.
\end{split}
\end{equation}
To obtain \eqref{def:wl1} we only need to show that
$v_t(y)=u_t(F_t(y))$, which follows at once
from \eqref{eqn:velocity1}, \eqref{eq:def F}, and \eqref{eqn:deriv-lim}.

\end{proof}

\section{Open problems}\label{sect:open pbs}

In this short section we point out some open problems. The first one
is of course uniqueness for the Cauchy problem, both at the level of
\eqref{eqn:dualsystem} and at the level of \eqref{eqn:SGsystem2}.
Let us point out that \emph{a priori} the two problems are not equivalent, because we proved that
solutions to \eqref{eqn:dualsystem} induce solutions to
\eqref{eqn:SGsystem2}, but at the moment the converse implication is only formal
(see the Appendix).

Another open question is the uniqueness of the regular Lagrangian
flow associated to $u$. Uniqueness is known, thanks to the results
in \cite{AM1}, for the flow $G$ in the dual variables with velocity
$U_t(y)=J(y-\nabla P_t^*(y))$; actually, in light of the $L\log^kL$
Sobolev regularity of $U$, even the quantitative stability results
of \cite{crippade1} are by now available for $G$. We were able in
the previous section to prove that flows $G_t$ of $U$ induce flows
$F_t$ of $u$, via the transformation $F_t= \nabla P_t^*\circ G_t\circ\nabla
P_0$. However,
our proof used the boundedness of $U$, an information we do not have
when we try to reverse the implication, namely that regular
Lagrangian flows $F$ of $u$ induce regular Lagrangian flows $G$ of
$U$ via the transformation $G=\nabla P_t\circ F_t\circ\nabla P_0^*$.
This question could be
settled, at least in the class of measure-preserving Lagrangian
flows, if the following conjecture had a positive answer:

\smallskip
\noindent {\bf Conjecture.} {\sl Let $f\in
W^{1,1}((0,T)\times\T^2;\R^2)\cap C([0,T]\times\T^2;\R^2)$, and let $H_t$ be a measure-preserving
Lagrangian flow relative to $b$. Assume that
\begin{equation}\label{cancellation}
[\partial_t f_t](H_t(x))+[\nabla f_t](H_t(x))b_t(H_t(x))\in
L^1(0,T)\qquad\text{for a.e. $x\in\T^2$.}
\end{equation}
Then for a.e. $x\in\T^2$ the map $t\mapsto f_t(H_t(x))$ is
absolutely continuous.} \smallskip

In our case, $f=\nabla P$ and $H_t$ is a measure-preserving flow
associated to $b=u$; with these choices, the term in
\eqref{cancellation} is equal to $U_t(x)$, so it is even bounded,
even though the summands in the expression might be unbounded.

We remark that if we assume that $f\in W^{1,q}$ for some $q>1$, and that
$$
\int_{\T^2}\int_0^T\biggl|\frac{d}{dt}H_t(x)\biggr|^p\,dt\,dx=
\int_{\T^2}\int_0^T\bigl|b_t(x)\bigr|^p\,dt\,dx<\infty, \qquad p=\frac{q}{q-1},
$$
then a simple approximation argument
based on convolving $f$ with smooth convolution kernels,
as the one used in the proof of Theorem \ref{thm:RLF},
provides a positive answer to the above conjecture. (This result
can also be seen
as a particular case of the general theory of weak gradients and absolute
continuity along curves recently developed in \cite{AM3,AM4}. However,
if $f$ is not continuous, one needs to replace $f$
with a suitable ``precise representative'' in its Lebesgue equivalence class.)
Observe that, in this latter case, \eqref{cancellation} is automatically satisfied
by Young inequality.

\appendix
\section{From physical to dual variables}

For completeness, we formally show how the dual equation
\eqref{eqn:dualsystem} is derived from system \eqref{eqn:SGsystem2}.
Taking into account the definition of $P_t$, the identities $J^2 = -
{Id}$, $\nabla p_t(y)+y=\nabla P_t(y)$, $\nabla^2
p_t(y)+{Id}=\nabla^2 P_t(y)$ and the fact that $u_t$ is
divergence-free, for every test function $\varphi$ we obtain
\begin{equation*}
\begin{split}
\frac{d}{dt} \int_{\T^2} \varphi(x)\, d \rho_t(x) &
=\frac{d}{dt}\int_{\T^2}\varphi(\nabla P_t(y))\, dy
=\int_{\T^2} \nabla\varphi(\nabla P_t(y))\cdot \frac{d}{dt} \nabla p_t(y)  \, dy \\
& =  -\int_{\T^2} \nabla\varphi(\nabla P_t(y)) \cdot
\Big\{( \nabla^2p_t(y) +{Id}) u_t(y) -J\nabla p_t(y)\Big\} \, dy \\
& =  -\int_{\T^2} \nabla \big[ \varphi(\nabla P_t(y)) \big] \cdot u_t(y) \, dy
+ \int_{\T^2} \nabla\varphi(\nabla P_t(y)) \cdot  J(\nabla P_t(y)-y) \, dy \\
& =  \int_{\T^2} \nabla\varphi(x) \cdot J( x - \nabla P_t^*(x)) \, d\rho_t(x)
=  \int_{\T^2} \nabla\varphi(x)\cdot U_t(x) \, d\rho_t(x).
\end{split}
\end{equation*}

Notice that this formal derivation holds independently of $u$
(only the divergence-free condition of $u$ is needed), and that $u$
does not appear explicitly in \eqref{eqn:dualsystem}.

\end{document}